\documentclass[11pt,reqno]{amsart}
\usepackage{amsfonts, amsmath, amssymb, color}
\usepackage{mathrsfs}
\usepackage{amsthm}
\usepackage{booktabs}
\usepackage{graphicx}
\usepackage{enumitem}
\usepackage{float}
\usepackage{hhline}
\usepackage{soul, color, xcolor}
\usepackage{lipsum}
\usepackage{lscape}
\usepackage{subfig}
\usepackage{textcomp}
\usepackage{tikz}
\usetikzlibrary{decorations.pathmorphing,shapes,arrows,positioning}
\usepackage{xcolor}
\usepackage{young}
\usepackage{youngtab}
\usepackage{ytableau}
\usepackage{rotating}
\usepackage[colorlinks=true, linkcolor=blue, citecolor=magenta, urlcolor=blue, backref=page]{hyperref}
\usepackage{scalefnt}
\usetikzlibrary{calc}
\usepackage{a4wide}

\hypersetup{colorlinks,linkcolor=blue,urlcolor=cyan,citecolor=magenta}

\allowdisplaybreaks[4]
\theoremstyle{plain}
\newtheorem{thm}{Theorem}[section]
\newtheorem{lem}[thm]{Lemma}
\newtheorem{prop}[thm]{Proposition}
\newtheorem{cor}[thm]{Corollary}
\newtheorem{rem}[thm]{Remark}

\theoremstyle{definition}

\newtheorem{exmp}[thm]{Example}

\makeatletter

\newcommand{\Rmnum}[1]{\expandafter\@slowromancap\romannumeral #1@}
\makeatother
\newcommand{\la}{\lambda}

\numberwithin{equation}{section} \errorcontextlines=0

\numberwithin{equation}{section} \errorcontextlines=0

\def\b{\mathfrak{b}}

\begin{document}
\title{Dual Murnaghan-Nakayama rule for Hecke algebras in Type $A$}
\author{Naihuan Jing}
\address{Department of Mathematics, North Carolina State University, Raleigh, NC 27695, USA}
\email{jing@ncsu.edu}
\author{Ning Liu}
\address{Beijing International Center for Mathematical Research, Peking University, Beijing 100871, China}
\email{mathliu123@outlook.com}
\author{Yu Wu}
\address{Shenzhen International Center for Mathematics, Southern University of Science and Technology, Shenzhen, Guangdong 518055, China}
\email{wywymath@163.com}

\begin{abstract}
Let $\chi^{\lambda}_{\mu}$ be the value of the irreducible character $\chi^{\lambda}$ of the Hecke algebra of the symmetric group on the conjugacy class of type $\mu$. The usual Murnaghan-Nakayama rule provides an iterative algorithm based on reduction of the lower partition $\mu$. In this paper, we establish a dual Murnaghan-Nakayama rule for Hecke algebras of type $A$ using vertex operators by applying reduction to the upper partition $\lambda$. We formulate an explicit recursion of the dual Murnaghan-Nakayama rule by employing the combinatorial model of ``brick tabloids", which refines a previous result by two of us (J. Algebra 598 (2022), 24--47).

\end{abstract}
\maketitle


\section{Introduction}
The Murnaghan-Nakayama rule, introduced independently by Murnaghan and Nakayama, provides a combinatorial method for computing the irreducible characters of the symmetric group $\mathfrak S_n$ \cite{Mur,Nak1,Nak2}.
In 1991, Ram \cite{Ram} employed the quantum Schur-Weyl duality \cite{Jim} to prove the Frobenius type character formula for the Iwahori-Hecke algebra $H_n(q)$ of type $A$ (see also \cite{KV,KW}). Furthermore, by this Frobenius formula he formulated a combinatorial $q$-Murnaghan-Nakayama rule for computing the irreducible characters of $H_n(q)$. Ram's $q$-Murnaghan-Nakayama rule can also be proved by induced characters of Coxeter elements and Kostka numbers \cite{Pfe}. Since Ram's work, there have been various discussions and generalizations on the rule and characters of Hecke algebras \cite{Gek,Hal,JL2,Roi,Sho,van}.

Recently, two of us \cite{JL1}  used the vertex operator realization of Schur functions to revisit this $q$-Murnaghan-Nakayama rule for the irreducible characters of $H_n(q)$. Additionally, in a dual picture, an iterative formula on upper partitions was obtained, which can be viewed  as the dual Murnaghan-Nakayama rule. It is worth noting that the dual version of the Murnaghan-Nakayama rule relies on the transition coefficient $C_{m,\rho}$ between the elementary symmetric function $e_m(x)$ and the generalized complete symmetric function $q_{\rho}(x;t)$ (see \cite[Theorem 2.14]{JL1}). 

The goal of this paper is to give a combinatorial interpretation of the coefficient $C_{m,\rho}$. More precisely, we express $C_{m,\rho}$ as a sum of rational functions in $q$ over all brick tabloids of size $m$; see Lemma \ref{l:S*} and Corollary \ref{c:Cmrho}. This leads to a refinement of the dual Murnaghan-Nakayama rule obtained in \cite{JL1}. In contrast with the dual recursion in \cite{JL1}, which involves the transition coefficients from $e_m$ to the basis $\{q_{\rho}\}$, the present formulation makes these coefficients explicit by means of brick tabloids. In this sense, the new rule is more direct and more transparent from a combinatorial point of view.

The resulting formula yields an explicit recursive algorithm for computing irreducible characters by reducing the upper partition $\lambda$ to $\lambda^{[1]}$ at each step; see Theorem \ref{iterative}. This upper-partition recursion is particularly convenient in examples where $\lambda_1$ is large; see Examples \ref{ex:Cmrho}, \ref{ex:example6}, and \ref{ex:811}.

The paper is organized as follows. In Section \ref{s:Pre}, we review partitions, brick tabloids, symmetric functions, the Frobenius character formula for $H_n(q)$, and the vertex operator realization of Schur functions. In Section \ref{s:OV}, we give a combinatorial interpretation of $C_{m,\rho}$, derive the refined dual Murnaghan-Nakayama rule, and present several examples.



\section{Preliminaries}\label{s:Pre}
\subsection{Partitions, brick tabloids and symmetric functions}
We first recall some basic terminologies about partitions and brick tabloids, mainly following \cite{Mac} and \cite{ER} respectively. A {\em partition} $\la=(\la_1,\la_2,\ldots)$ of {\em weight} $n$, denoted by $\la\vdash n$, is a finite sequence of descending nonnegative integers such that $|\la|=\sum_i\la_i=n$, where nonzero $\la_i$ are the {\em parts} of $\la$, the {\em length} $\ell(\la)$ is the number of the parts. When the parts are not necessarily descending, $\la$ is called a {\em composition} of $n$, denoted by $\la\vDash n$. We denote by $m_i(\la)$ the multiplicity of part $i$ in $\la$. The set of all partitions (resp. compositions) will be denoted by $\mathscr P$ (resp. $\mathscr C$). We also denote the set of all partitions (resp. compositions) of weight $n$ by $\mathscr{P}_n$ (resp. $\mathscr{C}_n$). When $\mu=(\mu_1,\ldots,\mu_s)$ is fixed, we will also use
\begin{align}
    \mathscr C_{\mu}:=\{\,\tau=(\tau_1,\ldots,\tau_s)\in \mathbb Z_{\ge 0}^{\,s}\mid 0\le \tau_i\le \mu_i \text{ for all } i\,\}.
\end{align}
For $\tau\in\mathscr C_{\mu}$, we write $|\tau|=\sum_i\tau_i$, and $\ell(\tau)$ denotes the number of nonzero parts of $\tau$.
For two partitions (resp.
compositions) $\mu,\nu$, we call $\tau\subset\mu$ if $\tau_i\leq \mu_i$ for all $i\geq1$. 
For a partition $\la=(\la_1,\la_2,\ldots,\la_\ell)$, we define 
\begin{align}
    \la^{[i]}=(\la_{i+1},\la_{i+2},\ldots,\la_\ell), \quad\text{for $1\leq i\leq \ell-1$.}
\end{align}

A {\em brick tabloid} of $n$ is a row of $n$ boxes, formed by the disjoint union of some bricks with lengths $b_i$ such that $\sum_{i}b_i=n$. For instance, the following is a brick tabloid of $10$.
\begin{figure}[H]
\begin{tikzpicture}[scale=1]
    \coordinate (Origin)   at (0,0);
    \coordinate (XAxisMin) at (0,0);
    \coordinate (XAxisMax) at (11,0);
    \coordinate (YAxisMin) at (0,0);
    \coordinate (YAxisMax) at (0,-11);
    \draw [thin, black] (0,0) -- (10,0);
    \draw [thin, black] (0,1) --(10,1);
    \draw [thin, black] (0,0) -- (0,1);
    \draw [thin, black] (1,0) -- (1,1);
    \draw [thin, black] (2,0) -- (2,1);
    \draw [thin, black] (3,0) -- (3,1);
    \draw [thin, black] (4,0) -- (4,1);
    \draw [thin, black] (5,0) -- (5,1);
     \draw [thin, black] (6,0) -- (6,1);
     \draw [thin, black] (7,0) -- (7,1);
     \draw [thin, black] (8,0) -- (8,1);
     \draw [thin, black] (9,0) -- (9,1);
      \draw [thin, black] (10,0) -- (10,1);
      \filldraw [fill=green, fill opacity=0.6, rounded corners]
    (0.5,0.3) rectangle (4.5,0.7) (5.5,0.3) rectangle (6.5,0.7) (7.5,0.3) rectangle (9.5,0.7);
     \end{tikzpicture}
    \caption{A brick tabloid of $10$ with $b_1=5$, $b_2=2$ and $b_3=3$.}\label{f:1}
    \end{figure}
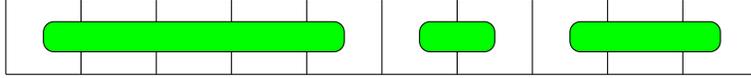

We denote by $\mathscr{B}_n$ the set of all brick tabloids of $n$. For a given brick tabloid $\b$, let $\ell(\b)$ be the number of bricks in $\b$ and $r_i(\b)$ the total number of boxes in the first $i$ bricks of $\b$, i.e., $r_i(\b):=\sum_{j\leq i}b_j$. We assume for convenience that $\varnothing$ is a brick tabloid of $0$ and $\ell(\varnothing)=0$. For a partition $\la$, a brick tabloid $\b$ is of type $\la$ if the partition obtained by rearranging the bricks of $\b$ is exactly $\la$. The brick tabloid in Fig. \ref{f:1} is of type $(5,3,2)$.  

Let $\Lambda_{\mathbb C}$ be the ring of symmetric functions over $\mathbb C$ in the variables $x_1,x_2,\ldots$. Let $p_r=\sum_i x_i^r$ be the $r$-th power-sum symmetric function. Then $\{p_{\lambda}=p_{\lambda_1}p_{\lambda_2}\cdots\mid \lambda\in\mathscr P\}$ forms a $\mathbb C$-basis of $\Lambda_{\mathbb C}$. Clearly, $p_r$ has degree $r$. In particular, if $\Lambda_{\mathbb C}^n$ denotes the homogeneous component of degree $n$, then $\{p_{\lambda}\mid \lambda\in\mathscr P_n\}$ forms an orthogonal basis of $\Lambda_{\mathbb C}^n$ with respect to
\begin{align}\label{e:inner}
    \langle p_\lambda,p_\mu \rangle=z_{\lambda}\delta_{\lambda\mu},
\end{align}
where $z_{\lambda}:=\prod_{i\geq 1}i^{m_i(\lambda)}m_i(\lambda)!$.
For $\lambda\in\mathscr P_n$, the Schur function $s_{\lambda}(x)$ is given by
\begin{align}
    s_\lambda(x)=\sum_{\mu\vdash n}\omega^{\lambda}(\mu)z_{\mu}^{-1}p_{\mu}.
\end{align}
Where $\omega^{\la}(\mu)$ is the irreducible character value of the symmetric group $\mathfrak S_n$ labled by $\la$ acting on the conjugacy class of type $\mu$.

It is well known that $\{s_{\lambda}\mid \lambda\in\mathscr P_n\}$ forms an orthonormal basis of $\Lambda_{\mathbb C}^n$ with respect to \eqref{e:inner}, namely,
\begin{align}\label{e:orthogonal}
    \langle s_\lambda,s_\mu \rangle=\delta_{\lambda\mu}.
\end{align}
The generalized complete symmetric function (that is, the one-row Hall-Littlewood function) $q_n(x;t)$ with parameter $t$ is defined by
\begin{equation}\label{e:generating-q}
Q(z;t)=\exp\left(\sum_{n=1}^{\infty}\frac{1-t^{n}}{n}p_nz^n\right)=\sum_{n=0}^{\infty}q_n(x;t)z^n.
\end{equation}
More explicitly,
\begin{align}
    q_n(x;t)=\sum_{\lambda\vdash n}\frac{1}{z_{\lambda}(t)}p_{\lambda},
\end{align}
where $z_{\lambda}(t)=z_{\lambda}/\prod_i(1-t^{\lambda_i})$. For any partition $\lambda$, we define
\begin{align}
    q_{\lambda}(x;t):=q_{\lambda_1}(x;t)q_{\lambda_2}(x;t)\cdots,
\end{align}
and the set $\{q_{\lambda}(x;t)\mid \lambda\in\mathscr P\}$ forms a $\mathbb C(t)$-basis of $\Lambda_{\mathbb C(t)}$.

\subsection{Frobenius character formula for Hecke algebras in type $A$}

The Iwahori--Hecke algebra $H_n(q)$, where $q$ is a fixed complex parameter, is the unital associative algebra generated by $T_1,T_2,\ldots,T_{n-1}$ subject to the relations
\begin{align}
    \begin{split}
        T_iT_j&=T_jT_i, \quad \mbox{if $|i-j|\geq 2$,}   \\
        T_iT_{i+1}T_i&=T_{i+1}T_iT_{i+1},\\
        T_i^2&=(q-1)T_i+q.
    \end{split}
\end{align}

For each $w\in S_n$ denote 
$T_w=T_{i_1}T_{i_2}\ldots T_{i_k}$ for any reduced expression of $w$. It is well-known that $T_w$ is independent from the choice of reduced expressions of $w$, and $\{T_w|w\in S_n\}$ forms a linear basis of $H_n(q)$.
Now for each partition $\mu=(\mu_1\ldots\mu_l)\vdash n$, let $T_{\gamma_{\mu}}=T_{\sigma_{\mu_1}}\ldots T_{\sigma_{\mu_l}}$ be the standard elements associated with the partition $\mu$. Here
$T_{\sigma}$ is the element attached to the cyclic permutation $\sigma$, and $T_{\sigma_{\mu_i}}$ are supported on disjoint indices $\mu_1+\cdots+\mu_{i-1}+1, \mu_1+\cdots+\mu_{i-1}+2, \ldots, \mu_1+\cdots+\mu_{i}$.

The complex irreducible representations of $H_n(q)$ are labeled by partitions of $n$. The trace of an irreducible representation is also referred to as its character. Ram \cite{Ram} proved a Frobenius-type character formula for $H_n(q)$ in terms of one-row Hall-Littlewood functions and Schur functions by using quantum Schur-Weyl duality. Any irreducible character $\chi^{\lambda}$ of $H_n(q)$ is completely determined by its values on the standard elements $T_{\gamma_{\mu}}$ (\(\mu\vdash n\)); see \cite{Ca,Ram}. We denote the value of $\chi^{\lambda}$ on $T_{\gamma_{\mu}}$ by $\chi^{\lambda}_{\mu}(q)$.

Introduce the modification
\begin{align}
    \widetilde q_r(x;t):=\frac{t^r}{t-1}q_r(x;t^{-1}),
    \qquad
    \widetilde q_{\lambda}(x;t):=\prod_{i=1}^{\ell(\lambda)}\widetilde q_{\lambda_i}(x;t)
    =\frac{t^{|\lambda|}}{(t-1)^{\ell(\lambda)}}q_{\lambda}(x;t^{-1}).
\end{align}

\begin{prop}\cite{Ram}
For each $\mu\vdash n$, one has
\begin{align}
    \widetilde q_{\mu}(x;q)=\sum_{\lambda\vdash n}\chi^{\lambda}_{\mu}(q)s_{\lambda}(x).
\end{align}
\end{prop}

Thanks to \eqref{e:orthogonal}, it follows that
\begin{align}\label{formula}
    \chi^{\lambda}_{\mu}(q)=\langle \widetilde q_{\mu}(x;q),s_{\lambda}(x) \rangle
    =\frac{q^n}{(q-1)^{\ell(\mu)}}\langle q_{\mu}(x;q^{-1}),s_\lambda(x) \rangle.
\end{align}

\subsection{Vertex operator realization of Schur functions}
Let $\mathfrak{h}$ be the infinite dimensional Heisenberg algebra over $\mathbb{C}$ with generators $h_n$, $n\in\mathbb{Z}\setminus\{0\}$ and the central element $c$ subject to the relation:
\begin{align}\label{e:hrel}
[h_m, h_n]=\delta_{m,-n}m\cdot c.
\end{align}

As is well known, there is a basic representation realized on the space $V=Sym(\mathfrak{h}^-)$ for the algebra $\mathfrak{h}$, where $V$ denotes the symmetric algebra generated by the elements $h_{-n}$, $n\in\mathbb{N}$. Explicitly, for any $n\in\mathbb{N}$, the element $h_{-n}$ acts as the multiplication by $h_{-n}$, while $h_n$ acts as the differentiation operator $n\frac{\partial}{\partial h_{-n}}$, which then satisfy the relation \eqref{e:hrel} with $c=1$.

Define the canonical inner product on $V$ by
\begin{align}\label{e:inner1}
\langle h_{-\la}, h_{-\mu} \rangle_{V}=\delta_{\la\mu}z_{\la},
\end{align}
where $h_{-\la}:=h_{-\la_1}h_{-\la_2}\cdots$

Therefore, $h_{-n}$ and $h_{n}$ are dual to each other with respect to \eqref{e:inner1}. Namely,
\begin{align}\label{e:dual1}
\langle h_{-n}f, g \rangle_{V}=\langle f, h_ng \rangle_{V} \quad\text{for any $f,g\in V$}.
\end{align}

Define the vertex operators $X_n$ and $X^*_n$: $V\longrightarrow V[[z, z^{-1}]]$ by
\begin{align}\label{e:Schurop1}
X(z)&=\mbox{exp} \left( \sum\limits_{n\geq 1} \dfrac{1}{n}h_{-n}z^{n} \right) \mbox{exp} \left( -\sum \limits_{n\geq 1} \dfrac{1}{n}h_{n}z^{-n} \right)=\sum_{n\in\mathbb{Z}}X_nz^n,\\ \label{e:Schurop*1}
X^*(z)&=\mbox{exp} \left(-\sum\limits_{n\geq 1} \dfrac{1}{n}h_{-n}z^{n} \right) \mbox{exp} \left(\sum \limits_{n\geq 1} \dfrac{1}{n}h_{n}z^{-n} \right)=\sum_{n\in\mathbb{Z}}X^*_nz^{-n}.
\end{align}
Recall that $\Lambda_{\mathbb{C}}$ denotes the space of symmetric functions. The {\it characteristic mapping} $\iota$ from $V$ to $\Lambda_{\mathbb{C}}$ is defined by
\begin{align}
\iota(h_{-\la}):=p_{\la_1}\cdots p_{\la_\ell}=p_{\la} \quad \text{for any partition $\la$}.
\end{align}
By definition, $\iota(h_n)=\iota(n\frac{\partial}{\partial h_{-n}})=n\frac{\partial}{\partial p_n}:=p^*_n$. It is shown \cite[Theorem 3.6]{J2} that $\iota$ is an isometric isomorphism between $V$ and $\Lambda_{\mathbb{C}}$.

It follows from \eqref{e:dual1} that
\begin{align}\label{e:star}
\langle p_{n}f, g \rangle=\langle f, p_n^*g \rangle \quad\text{for any $f,g\in \Lambda_{\mathbb{C}}$}.
\end{align}
Note that $*$ is $\mathbb C$-linear and  anti-involutive.

Under the characteristic mapping, the images of $X(z)$ and $X^*(z)$ (denoted by $S(z)$ and $S^*(z)$ respectively) are given by
\begin{align}\label{e:Schurop2}
\iota(X(z))=S(z)&=\mbox{exp} \left( \sum\limits_{n\geq 1} \dfrac{1}{n}p_{n}z^{n} \right) \mbox{exp} \left( -\sum \limits_{n\geq 1} \dfrac{\partial}{\partial p_n}z^{-n} \right)=\sum_{n\in\mathbb{Z}}S_nz^n,\\ \label{e:Schurop*2}
\iota(X^*(z))=S^*(z)&=\mbox{exp} \left(-\sum\limits_{n\geq 1} \dfrac{1}{n}p_{n}z^{n} \right) \mbox{exp} \left(\sum \limits_{n\geq 1} \dfrac{\partial}{\partial p_n}z^{-n} \right)=\sum_{n\in\mathbb{Z}}S^*_nz^{-n}.
\end{align}

Let us recall the vertex operator realization of Schur functions and the relations between $S_{n}$ and $S^*_{m}$.
\begin{prop}\label{p:VOschur}\cite{J1} 
\begin{enumerate}

\item For any composition $\mu=(\mu_{1},\ldots,\mu_{k})$, define
\begin{align}
    S_{\mu}.1:=S_{\mu_{1}}\cdots S_{\mu_{k}}\cdot 1.
\end{align}
Then $S_{\mu}.1=s_{\mu}$. In general, $s_{\mu}=0$ or $\pm s_{\lambda}$ for a partition $\lambda$ such that $\lambda\in \mathfrak{S}_{\ell}(\mu+\delta)-\delta$, where $\delta=(\ell-1,\ell-2,\ldots,1,0)$ and $\ell=\ell(\mu)$. Moreover,
\begin{align}
    S_{-n}.1=\delta_{n,0},
    \qquad
    S^{*}_{n}.1=\delta_{n,0},
    \qquad (n\geq 0).
\end{align}

\item  The components of $S(z)$ and $S^{*}(z)$ obey the following commutation relations:
\begin{align}
S_{m}S_{n}+S_{n-1}S_{m+1}&=0,\\
S^{*}_{m}S^{*}_{n}+S^{*}_{n+1}S^{*}_{m-1}&=0,\\
S_{m}S^{*}_{n}+S^{*}_{n-1}S_{m-1}&=\delta_{m,n}.
\end{align}
\end{enumerate}
\end{prop}
\begin{rem}\label{r:Sr1}
For $r\in\mathbb Z$, we write $S_r.1:=S_r\cdot 1$, namely, the action of the operator $S_r$ on the constant symmetric function $1$. Similarly, we write $S_r^*.1:=S_r^*\cdot 1$.
By \eqref{e:Schurop2}, one has
\begin{align}
    S_r.1=
    \begin{cases}
        h_r(x), & \text{if } r\ge 0,\\
        0, & \text{if } r<0,
    \end{cases}
\end{align}
and by \eqref{e:Schurop*2},
\begin{align}
    S_{-k}^*.1=
    \begin{cases}
        (-1)^k e_k(x), & \text{if } k\ge 0,\\
        0, & \text{if } k<0.
    \end{cases}
\end{align}
\end{rem}


    


We can rewrite the character formula \eqref{formula} in terms of vertex operators as follows
\begin{align}\label{formula 2}
\begin{split}
    \chi^{\la}_{\mu}(q)&=\dfrac{q^{|\mu|}}{(q-1)^{\ell(\mu)}}\langle q_{\mu}(x;q^{-1}),S_{\la}.1\rangle \quad \text{(by Proposition \ref{p:VOschur} (1))}\\
    &=\dfrac{q^{|\mu|}}{(q-1)^{\ell(\mu)}}\langle S^*_{\la_1}q_{\mu}(x;q^{-1}),S_{\la^{[1]}}.1\rangle \quad \text{(by  \eqref{e:star})}.
    \end{split}
\end{align}

The following result is first obtained in \cite{JL1} by induction. Here we provide a direct proof.
\begin{prop}\label{S*q}
For a partition $\mu$ and any integer $k$, one has
\begin{align}
    S_k^*q_{\mu}(x;t)=\sum_{\tau\in\mathscr C_\mu}(1-t)^{\ell(\tau)}q_{\mu-\tau}(x;t)S_{k-|\tau|}^*.1,
\end{align}
where
\begin{align}
    q_{\mu-\tau}(x;t):=\prod_{i=1}^{\ell(\mu)}q_{\mu_i-\tau_i}(x;t).
\end{align}
\end{prop}

\begin{proof}
First,
\begin{align*}
S^*(z)Q(w;t)
&=Q(w;t)S^*(z)\exp\left[\sum_{n\geq1}\frac{\partial}{\partial p_n}z^{-n},\sum_{n\geq1}\frac{1-t^n}{n}p_nw^n\right]\\
&=Q(w;t)S^*(z)\frac{z-tw}{z-w}\\
&=Q(w;t)S^*(z)\left(1+(1-t)\sum_{j\geq1}\left(\frac{w}{z}\right)^j\right).
\end{align*}
Write $\mu=(\mu_1,\ldots,\mu_s)$, where $s=\ell(\mu)$. Repeating the above identity gives
\begin{align*}
S^*(z)\prod_{i=1}^{s}Q(w_i;t)
&=\prod_{i=1}^{s}Q(w_i;t)S^*(z)\prod_{i=1}^{s}\left(1+(1-t)\sum_{j\geq1}\left(\frac{w_i}{z}\right)^j\right)\\
&=\prod_{i=1}^{s}Q(w_i;t)S^*(z)
\sum_{j_1,\ldots,j_s\geq0}(1-t)^{\#\{i\mid j_i>0\}}
\frac{w_1^{j_1}\cdots w_s^{j_s}}{z^{j_1+\cdots+j_s}}.
\end{align*}
Now let both sides act on $1$ and take the coefficient of
\[
z^{-k}w_1^{\mu_1}\cdots w_s^{\mu_s}.
\]
The left-hand side gives $S_k^*q_{\mu}(x;t)$. On the right-hand side, choosing $j_i=\tau_i$ with $0\leq \tau_i\leq \mu_i$ contributes
\[
(1-t)^{\ell(\tau)}q_{\mu-\tau}(x;t)S_{k-|\tau|}^*.1.
\]
Summing over all $\tau\in\mathscr C_\mu$ completes the proof.
\end{proof}


\section{Main results}\label{s:OV}
In this section, we will present our main result: a new dual Murnaghan-Nakayama rule for the Hecke algebra in type $A$. Before that, we need to establish the combinatorial interpretation for the coefficient $C_{m,\rho}$ defined by 
\begin{align}
    e_m(x)=\sum_{\rho\vdash m}C_{m,\rho}q_{\rho}(x;t).
\end{align}



\begin{lem}\label{l:S*}
For any integer $n$, with the convention $\mathscr B_n=\varnothing$ for $n<0$, one has
\begin{align}\label{e:S*}
S^*_{-n}.1=\sum_{\b\in \mathscr{B}_n} \prod_{i=1}^{\ell(\b)}\frac{1}{q^{-r_i(\b)}-1}\,q_{\b}(x;q^{-1}),
\end{align}
where
\[
q_{\b}(x;q^{-1}):=q_{b_1}(x;q^{-1})q_{b_2}(x;q^{-1})\cdots
\]
for $\b=(b_1,b_2,\ldots)$.
\end{lem}

\begin{proof}
If $n<0$, then $S^*_{-n}.1=0$ by Remark \ref{r:Sr1}, and the right-hand side is empty. If $n=0$, both sides are equal to $1$. Hence it suffices to consider $n>0$. We abbreviate $q_j(x;q^{-1})$ by $q_j$.

By \eqref{e:generating-q} and \eqref{e:Schurop*2},
\begin{align}
S^*(q^{-1}z).1=Q(z;q^{-1})S^*(z).1.
\end{align}
Comparing the coefficients of $z^n$ on both sides gives
\begin{align}
(q^{-n}-1)S^*_{-n}.1=q_1S^*_{-n+1}.1+q_2S^*_{-n+2}.1+\cdots+q_{n-1}S^*_{-1}.1+q_n.
\end{align}
Equivalently,
\begin{align}
S^*_{-n}.1=\frac{1}{q^{-n}-1}\left(q_1S^*_{-n+1}.1+q_2S^*_{-n+2}.1+\cdots+q_{n-1}S^*_{-1}.1+q_n\right).
\end{align}
Assume inductively that \eqref{e:S*} holds for all positive integers smaller than $n$. Then
\begin{align*}
S^*_{-n}.1
=&\frac{1}{q^{-n}-1}\Bigg(
q_1\sum_{\b\in \mathscr B_{n-1}}\prod_{i=1}^{\ell(\b)}\frac{1}{q^{-r_i(\b)}-1}q_{\b}
+\cdots
+q_{n-1}\sum_{\b\in \mathscr B_{1}}\prod_{i=1}^{\ell(\b)}\frac{1}{q^{-r_i(\b)}-1}q_{\b}
+q_n\Bigg).
\end{align*}
For each $1\leq j\leq n$, multiplication by $q_j/(q^{-n}-1)$ appends a last brick of length $j$. Hence
\begin{align}
\frac{q_j}{q^{-n}-1}\sum_{\b\in \mathscr B_{n-j}}\prod_{i=1}^{\ell(\b)}\frac{1}{q^{-r_i(\b)}-1}q_{\b}
=
\sum_{\b\in \mathscr B_n^{(j)}}\prod_{i=1}^{\ell(\b)}\frac{1}{q^{-r_i(\b)}-1}q_{\b},
\end{align}
where $\mathscr B_n^{(j)}$ denotes the set of brick tabloids of $n$ whose last brick has length $j$. Since
\[
\mathscr B_n=\bigsqcup_{j=1}^{n}\mathscr B_n^{(j)},
\]
the desired identity follows.
\end{proof}

\begin{cor}\label{c:Cmrho}
Suppose $m\geq 0$ and $\rho\vdash m$. Then
\begin{align}
   C_{m,\rho}=(-1)^m\sum_{\b\in\mathscr{B}_{\rho}}\prod_{i=1}^{\ell(\b)}\frac{1}{t^{r_i(\b)}-1},
\end{align}
where $\mathscr{B}_{\rho}$ denotes the set of brick tabloids of type $\rho$.
\end{cor}

\begin{proof}
By Remark \ref{r:Sr1}, we have
\begin{align}
    e_m(x)=(-1)^mS^*_{-m}.1.
\end{align}
Applying Lemma \ref{l:S*} with $q=t^{-1}$, we obtain
\begin{align}
    e_m(x)=(-1)^m\sum_{\b\in\mathscr B_m}\prod_{i=1}^{\ell(\b)}\frac{1}{t^{r_i(\b)}-1}\,q_\b(x;t).
\end{align}
Grouping the above sum according to the type $\rho$ of $\b$, we get
\begin{align}
e_m(x)=\sum_{\rho\vdash m}\left[(-1)^m\sum_{\b\in\mathscr B_{\rho}}\prod_{i=1}^{\ell(\b)}\frac{1}{t^{r_i(\b)}-1}\right]q_{\rho}(x;t).
\end{align}
Comparing this with
\[
e_m(x)=\sum_{\rho\vdash m}C_{m,\rho}q_{\rho}(x;t)
\]
yields the claimed formula.
\end{proof}

\begin{exmp}\label{ex:Cmrho}
The first coefficients $C_{m,\rho}$ can be read off directly from Corollary \ref{c:Cmrho}. For $m=2$, the brick tabloids of $2$ are $(2)$ and $(1,1)$, so
\begin{align}
e_2(x)=\frac{1}{t^2-1}q_2(x;t)+\frac{1}{(t-1)(t^2-1)}q_{1,1}(x;t).
\end{align}
For $m=3$, the brick tabloids are $(3)$, $(2,1)$, $(1,2)$, and $(1,1,1)$. Grouping them by type gives
\begin{align}
e_3(x)
=&-\frac{1}{t^3-1}q_3(x;t) \notag\\
&-\left(\frac{1}{(t^2-1)(t^3-1)}+\frac{1}{(t-1)(t^3-1)}\right)q_{2,1}(x;t) \notag\\
&-\frac{1}{(t-1)(t^2-1)(t^3-1)}q_{1,1,1}(x;t).
\end{align}
\end{exmp}




\begin{thm}\label{iterative}
For $\lambda,\mu\vdash n$, the dual Murnaghan-Nakayama rule for the irreducible characters of the Hecke algebra of type $A$ is
\begin{align}
\chi^{\lambda}_{\mu}(q)
=\sum_{\tau\in \mathscr C_{\mu}}q^{\lambda_1-\ell(\tau)}
\sum_{\b\in\mathscr{B}_{|\tau|-\lambda_1}}(q-1)^{\ell(\tau)-\ell(\mu)+\ell(\mu-\tau)+\ell(\b)}
\prod_{j=1}^{\ell(\b)}\frac{1}{q^{-r_j(\b)}-1}\chi^{\lambda^{[1]}}_{(\mu-\tau)\cup\b}(q),
\end{align}
where $(\mu-\tau)\cup\b$ denotes the partition obtained by rearranging the nonzero parts of $\mu-\tau$ together with the brick lengths of $\b$.
\end{thm}

\begin{proof}
By \eqref{formula 2}, Proposition \ref{S*q}, and Lemma \ref{l:S*}, we have
\begin{align*}
\chi^{\lambda}_{\mu}(q)
&=\frac{q^n}{(q-1)^{\ell(\mu)}}\sum_{\tau\in\mathscr C_\mu}(1-q^{-1})^{\ell(\tau)}
\left\langle q_{\mu-\tau}(x;q^{-1})S^*_{\lambda_1-|\tau|}.1,S_{\lambda^{[1]}}.1\right\rangle \\
&=\frac{q^n}{(q-1)^{\ell(\mu)}}\sum_{\tau\in\mathscr C_\mu}(1-q^{-1})^{\ell(\tau)}
\sum_{\b\in\mathscr B_{|\tau|-\lambda_1}}
\prod_{j=1}^{\ell(\b)}\frac{1}{q^{-r_j(\b)}-1}
\left\langle q_{(\mu-\tau)\cup\b}(x;q^{-1}),S_{\lambda^{[1]}}.1\right\rangle.
\end{align*}
Now
\[
|(\mu-\tau)\cup\b|
=|\mu|-|\tau|+|\tau|-\lambda_1
=n-\lambda_1
=|\lambda^{[1]}|.
\]
Hence, by \eqref{formula},
\[
\left\langle q_{(\mu-\tau)\cup\b}(x;q^{-1}),S_{\lambda^{[1]}}.1\right\rangle
=
\frac{(q-1)^{\ell((\mu-\tau)\cup\b)}}{q^{\,n-\lambda_1}}
\chi^{\lambda^{[1]}}_{(\mu-\tau)\cup\b}(q).
\]
Since
\[
\ell((\mu-\tau)\cup\b)=\ell(\mu-\tau)+\ell(\b),
\]
substituting this into the previous formula and using
\[
1-q^{-1}=\frac{q-1}{q},
\]
we obtain
\begin{align*}
\chi^{\lambda}_{\mu}(q)
=\sum_{\tau\in \mathscr C_{\mu}}q^{\lambda_1-\ell(\tau)}
\sum_{\b\in\mathscr{B}_{|\tau|-\lambda_1}}(q-1)^{\ell(\tau)-\ell(\mu)+\ell(\mu-\tau)+\ell(\b)}
\prod_{j=1}^{\ell(\b)}\frac{1}{q^{-r_j(\b)}-1}\chi^{\lambda^{[1]}}_{(\mu-\tau)\cup\b}(q),
\end{align*}
which is the desired recursion.
\end{proof}

\begin{rem}\label{r:efficiency}
Theorem \ref{iterative} gives a genuine recursive algorithm. Each application replaces the upper partition $\lambda$ by $\lambda^{[1]}$, so after at most $\ell(\lambda)-1$ steps one reaches a one-row partition. For a one-row partition $(m)$, one has
\[
\chi^{(m)}_{\nu}(q)=q^{m-\ell(\nu)}.
\]
The main advantage of Theorem \ref{iterative} lies in its direct upper-partition recursion. In particular, when $\lambda_1$ is large, only terms with $|\tau|\ge \lambda_1$ contribute, and one application of the theorem reduces the computation from a character of $H_n(q)$ to a character of $H_{n-\lambda_1}(q)$. This feature is illustrated in Examples \ref{ex:example6} and \ref{ex:811}.
\end{rem}


\begin{exmp}\label{ex:example6}
Let $\lambda=(3,2,1)$ and $\mu=(4,2)\vdash 6$. By Theorem \ref{iterative}, only the cases $|\tau|=3,4,5,6$ contribute. When $|\tau|=6$, we have $\tau=\mu$ and
\begin{align*}
   \mathscr B_3= \left\{\begin{tikzpicture}[scale=1]
    \coordinate (Origin)   at (0,0);
    \coordinate (XAxisMin) at (0,0);
    \coordinate (XAxisMax) at (11,0);
    \coordinate (YAxisMin) at (0,0);
    \coordinate (YAxisMax) at (0,-11);
    \draw [thin, black] (0,0) -- (3,0);
    \draw [thin, black] (0,1) --(3,1);
    \draw [thin, black] (0,0) -- (0,1);
    \draw [thin, black] (1,0) -- (1,1);
    \draw [thin, black] (2,0) -- (2,1);
    \draw [thin, black] (3,0) -- (3,1);
      \filldraw [fill=green, fill opacity=0.6, rounded corners]
    (0.25,0.3) rectangle (0.75,0.7) (1.25,0.3) rectangle (1.75,0.7) (2.25,0.3) rectangle (2.75,0.7);
     \end{tikzpicture},\quad
     \begin{tikzpicture}[scale=1]
    \coordinate (Origin)   at (0,0);
    \coordinate (XAxisMin) at (0,0);
    \coordinate (XAxisMax) at (11,0);
    \coordinate (YAxisMin) at (0,0);
    \coordinate (YAxisMax) at (0,-11);
    \draw [thin, black] (0,0) -- (3,0);
    \draw [thin, black] (0,1) --(3,1);
    \draw [thin, black] (0,0) -- (0,1);
    \draw [thin, black] (1,0) -- (1,1);
    \draw [thin, black] (2,0) -- (2,1);
    \draw [thin, black] (3,0) -- (3,1);
      \filldraw [fill=green, fill opacity=0.6, rounded corners]
    (0.5,0.3) rectangle (1.5,0.7) (2.25,0.3) rectangle (2.75,0.7);
     \end{tikzpicture},\quad
      \begin{tikzpicture}[scale=1]
    \coordinate (Origin)   at (0,0);
    \coordinate (XAxisMin) at (0,0);
    \coordinate (XAxisMax) at (11,0);
    \coordinate (YAxisMin) at (0,0);
    \coordinate (YAxisMax) at (0,-11);
    \draw [thin, black] (0,0) -- (3,0);
    \draw [thin, black] (0,1) --(3,1);
    \draw [thin, black] (0,0) -- (0,1);
    \draw [thin, black] (1,0) -- (1,1);
    \draw [thin, black] (2,0) -- (2,1);
    \draw [thin, black] (3,0) -- (3,1);
      \filldraw [fill=green, fill opacity=0.6, rounded corners]
    (0.25,0.3) rectangle (0.75,0.7) (1.5,0.3) rectangle (2.5,0.7);
     \end{tikzpicture},\quad
      \begin{tikzpicture}[scale=1]
    \coordinate (Origin)   at (0,0);
    \coordinate (XAxisMin) at (0,0);
    \coordinate (XAxisMax) at (11,0);
    \coordinate (YAxisMin) at (0,0);
    \coordinate (YAxisMax) at (0,-11);
    \draw [thin, black] (0,0) -- (3,0);
    \draw [thin, black] (0,1) --(3,1);
    \draw [thin, black] (0,0) -- (0,1);
    \draw [thin, black] (1,0) -- (1,1);
    \draw [thin, black] (2,0) -- (2,1);
    \draw [thin, black] (3,0) -- (3,1);
      \filldraw [fill=green, fill opacity=0.6, rounded corners]
    (0.5,0.3) rectangle (2.5,0.7);
     \end{tikzpicture}\right\}.
\end{align*}
Therefore the corresponding contribution is
\begin{align*}
&\frac{q(q-1)^3}{(q^{-1}-1)(q^{-2}-1)(q^{-3}-1)}\chi^{(2,1)}_{(1^3)}(q) \\
&+\frac{q(q-1)^2}{(q^{-1}-1)(q^{-3}-1)}\chi^{(2,1)}_{(2,1)}(q)
+\frac{q(q-1)^2}{(q^{-2}-1)(q^{-3}-1)}\chi^{(2,1)}_{(2,1)}(q) \\
&+\frac{q(q-1)}{q^{-3}-1}\chi^{(2,1)}_{(3)}(q)=0,
\end{align*}
where we used $\chi^{(2,1)}_{(1^3)}(q)=2$, $\chi^{(2,1)}_{(2,1)}(q)=q-1$, and $\chi^{(2,1)}_{(3)}(q)=-q$. Similarly, the contributions corresponding to $|\tau|=5,4,3$ are
\[
2q^4-2q^3,\qquad -4q^4+7q^3-3q^2,\qquad 2q^4-6q^3+5q^2-q,
\]
respectively. Therefore,
\[
\chi^{(3,2,1)}_{(4,2)}(q)=-q^3+2q^2-q.
\]
In this example, one application of Theorem \ref{iterative} reduces the computation from $H_6(q)$ to characters of $H_3(q)$.
\end{exmp}

\begin{exmp}\label{ex:811}
Let $\lambda=(6,1,1)$ and $\mu=(2,2,2,2)\vdash 8$. This is precisely the kind of example for which Theorem \ref{iterative} is convenient, because $\lambda_1=6$, so one application reduces the problem to characters of $H_2(q)$. The relevant $H_2(q)$-characters are
\[
\chi^{(1,1)}_{(1,1)}(q)=1,\qquad \chi^{(1,1)}_{(2)}(q)=-1.
\]
Only the cases $|\tau|=6,7,8$ contribute.

If $|\tau|=6$, then $\mathscr B_{|\tau|-\lambda_1}=\mathscr B_0=\{\varnothing\}$. Among the $\tau\in\mathscr C_\mu$ with $|\tau|=6$, four are permutations of $(2,2,2,0)$, and each contributes
\[
q^{6-3}(q-1)^{3-4+1}\chi^{(1,1)}_{(2)}(q)=-q^3,
\]
while six are permutations of $(2,2,1,1)$, and each contributes
\[
q^{6-4}(q-1)^{4-4+2}\chi^{(1,1)}_{(1,1)}(q)=q^2(q-1)^2.
\]
Hence the total contribution for $|\tau|=6$ is
\[
-4q^3+6q^2(q-1)^2=6q^4-16q^3+6q^2.
\]

If $|\tau|=7$, then $\mathscr B_1=\{(1)\}$. There are four permutations of $(2,2,2,1)$, and each contributes
\[
q^{6-4}(q-1)^{4-4+1+1}\frac{1}{q^{-1}-1}\chi^{(1,1)}_{(1,1)}(q)=q^3(1-q).
\]
Thus the total contribution for $|\tau|=7$ is
\[
4q^3(1-q)=-4q^4+4q^3.
\]

If $|\tau|=8$, then $\tau=\mu$ and $\mathscr B_2=\{(1,1),(2)\}$. The two contributions are
\begin{align*}
q^{6-4}(q-1)^{4-4+0+2}\frac{1}{(q^{-1}-1)(q^{-2}-1)}\chi^{(1,1)}_{(1,1)}(q)
&=\frac{q^5}{q+1},\\
q^{6-4}(q-1)^{4-4+0+1}\frac{1}{q^{-2}-1}\chi^{(1,1)}_{(2)}(q)
&=\frac{q^4}{q+1},
\end{align*}
whose sum is $q^4$.

Combining the three cases, we obtain
\[
\chi^{(6,1,1)}_{(2,2,2,2)}(q)=3q^4-12q^3+6q^2.
\]
This example shows that when $\lambda_1$ is large, Theorem \ref{iterative} can reduce a character computation in $H_n(q)$ to one in a much smaller Hecke algebra in a single step. Compared with the formulation in \cite{JL1}, the present method is more direct in this example: the transition coefficients from $e_m$ to the basis $\{q_{\rho}\}$ are already encoded explicitly by Corollary \ref{c:Cmrho}, so the recursion can be carried out directly once the relevant brick tabloids are listed. This illustrates the advantage of the present combinatorial refinement.
\end{exmp}

\section*{Acknowledgments}
N.J. is partially supported by the Simons Foundation (No. MP-TSM-00002518) and National Natural Science Foundation of China (No. 12171303). 
N.L. is supported by China Postdoctoral Science Foundation (No. 2025M773076, No. GZC20252018) and Beijing Natural Science Foundation (No. 1264053). 
\bigskip


\begin{thebibliography}{AB99}
\bibitem[C86]{Ca} R.~W. ~Carter, {\em Representation theory of the 0-Hecke algebra}, J. Algebra 104(1986), 89--103.

 \bibitem[ER91]{ER} \"O. E\v{g}ecio\v{g}lu, J. B. Remmel, {\em Brick tabloids and the connection matrices between bases of symmetric functions}, Discrete Appl. Math. 34 (1991), no. 1-3, 107--120.

\bibitem[G99]{Gek} M. Geck, {\em The character table of the Iwahori-Hecke algebra of the symmetric group: Starkey’s rule}, C. R. Acad. Sci., Sér. 1 Math. 329 (1999) 361--366.

\bibitem[H95]{Hal} T. Halverson, {\em A $q$-rational Murnaghan-Nakayama rule}, J. Comb. Theory, Ser. A 71 (1995) 1--18.

\bibitem[J86]{Jim} M. Jimbo, {\em A $q$-analogue of $U(gl(N+1))$, Hecke algebra, and the Yang-Baxter equation}, Lett. Math. Phys. 11 (1986), no. 3, 247--252.


\bibitem[J91a]{J1} N. Jing, {\em Vertex operators, symmetric functions and the spin group $\Gamma _n$}, J. Algebra 138 (1991), 340--398.

\bibitem[J91b]{J2} N. Jing, {\em Vertex operators and Hall-Littlewood symmetric functions}, Adv. Math. 87 (1991), 226--248.

\bibitem[JL22]{JL1} N.~Jing, N.~Liu, {\em On irreducible characters of the Iwahori-Hecke algebra in type $A$}, J. Algebra. 598 (2022), 24--47.

\bibitem[JL23]{JL2} N.~Jing, N.~Liu, {\em Murnaghan-Nakayama rule and spin bitrace for the Hecke-Clifford algebra}, Int. Math. Res. Not. 19 (2023), 17060--17099.

\bibitem[KV89]{KV} A.~V.~Kerov, A.~M.~Vershik, {\it Characters and realizations of representations of an infinite-dimensional Hecke algebra, and knot invariants}, Sov. Math., Dokl. 38 (1989), 134--137.

\bibitem[KW92]{KW} R.~C.~King,  B.~G.~Wybourne, {\it Representations and traces of the Hecke algebras $H_n(q)$ of type $A_{n-1}$}, J. Math. Phys. 33 (1992), 4--14.

\bibitem[M98]{Mac}  I.~G.~Macdonald, {\it Symmetric functions and Hall polynomials}, 2nd. ed., Oxford University Press, Oxford, 1998.

\bibitem[M37]{Mur} F.~D.~Murnaghan, {\em The characters of the symmetric group}, Amer. J. Math. 59 (1937), 739--753.

\bibitem[N41a]{Nak1} T.~Nakayama, {\em On some modular properties of irreducible representatitons of a symmetric group $\Rmnum{1}$}, Jap. J. Math. 18 (1941), 89--108.

\bibitem[N41b]{Nak2} T.~Nakayama, {\em On some modular properties of irreducible representatitons of a symmetric group $\Rmnum{2}$}, Jap. J. Math. 17 (1941), 411--423.

\bibitem[P94]{Pfe} G. Pfeiffer, {\em Young characters on Coxeter basis elements of Iwahori-Hecke algebras and a Murnaghan-Nakayama formula},
J. Algebra 168 (1994), no. 2, 525--535.

 \bibitem[R91]{Ram} A.~Ram, {\it A Frobenius formula for the characters of the Hecke algebras}, Invent. Math. 106 (1991), 461--488.

\bibitem[R97]{Roi} Y. Roichman, {\em A recursive rule for Kazhdan-Lusztig characters}, Adv. Math. 129 (1997) 25--45.

\bibitem[S00]{Sho} T. Shoji, {\em A Frobenius formula for the characters of Ariki-Koike algebras}, J. Algebra 226 (2000) 818--856.


\bibitem[vJ91]{van} J. van der Jeugt, {\em An algorithm for characters of Hecke algebras $H_n(q)$ of type $A_{n-1}$}, J. Phys. A 24
(1991) 3719--3724.

\end{thebibliography}
\end{document}